\documentclass{amsart}
\usepackage{amsfonts, amsbsy, amsmath, amssymb}

\ifx\fiverm\undefined 
  \newfont\fiverm{cmr5} 
\fi

\newtheorem{thm}{Theorem}[section]
\newtheorem{lem}[thm]{Lemma}

\newtheorem{prop}[thm]{Proposition}

\newtheorem{rmk}[thm]{Remark}

\newtheorem{thm-con}[thm]{Theorem-Conjecture}
\numberwithin{equation}{section}

\theoremstyle{definition}

\allowdisplaybreaks

\newcommand{\f}{\Bbb F}

\newcommand\longequation[1]{\parbox[t]{\textwidth}{\raggedright$#1$}}

\begin{document}

\title[Normal Polynomials over Finite Fields]{A Criterion for the Normality of Polynomials over Finite Fields Based on Their Coefficients}

\author[Xiang-dong Hou]{Xiang-dong Hou}
\address{Department of Mathematics and Statistics,
University of South Florida, Tampa, FL 33620}
\email{xhou@usf.edu}




\begin{abstract}
An irreducible polynomial over $\Bbb F_q$ is said to be normal over $\Bbb F_q$ if its roots are linearly independent over $\Bbb F_q$. We show that there is a polynomial $h_n(X_1,\dots,X_n)\in\Bbb Z[X_1,\dots,X_n]$, independent of $q$, such that if an irreducible polynomial $f=X^n+a_1X^{n-1}+\cdots+a_n\in\Bbb F_q[X]$ is such that $h_n(a_1,\dots,a_n)\ne 0$, then $f$ is normal over $\Bbb F_q$. The polynomial $h_n(X_1,\dots,X_n)$ is computed explicitly for $n\le 5$ and partially for $n=6$. When $\text{char}\,\Bbb F_q=p$, we also show that there is a polynomial $h_{p,n}(X_1,\dots,X_n)\in\Bbb F_p[X_1,\dots,X_n]$, depending on $p$, which is simpler than $h_n$ but has the same property. These results remain valid for monic separable irreducible polynomials over an arbitrary field with a cyclic Galois group.
\end{abstract}

\maketitle

\section{Introduction}

Let $F$ be a field and $K$ be a finite Galois extension over $K$. If $a\in K$ is such that $\{\sigma(a):\sigma\in\text{Aut}(K/F)\}$ forms a basis of $K/F$, $a$ is called a {\em normal element} of $K$ over $F$ and $\{\sigma(a):\sigma\in\text{Aut}(K/F)\}$ is called a {\em normal basis} of $K$ over $F$. The existence of normal bases was proved a long time ago by Deuring \cite{Deuring-MA-1933}. (For a more accessible reference, see Lang \cite[\S VI.13]{Lang-2002}.) Normal bases have applications in many areas, especially in the theory and applications of finite fields.

Let $\f_q$ denote the finite field with $q$ elements. An irreducible polynomial $f\in\f_q[X]$ of degree $n$ is said to be {\em normal} over $\f_q$ if its roots are linearly independent over $\f_q$, i.e., the roots form a normal basis of $\f_{q^n}$ over $\f_q$. Many theoretic results and computational procedures concerning finite fields utilize normal bases. There is an extensive literature on normal bases of finite fields; see \cite{Gao-thesis-1993}, \cite[\S\S 5.2 -- 5.4]{Mullen-Panario-HF-2013}, and the references therein.

In this paper, we are interested in finding sufficient conditions for an irreducible polynomial over $\f_q$ to be normal. More precisely, we show that there is a polynomial $h_n(X_1,\dots,X_n)\in\Bbb Z[X_1,\dots,X_n]$, independent of $q$, such that if $f=X^n+a_1X^{n-1}+\cdots+a_n\in\f_q[X]$ is irreducible and $h_n(a_1,\dots,a_n)\ne 0$, then $f$ is normal over $\f_q$. The polynomial $h_n$ is computed explicitly for $n\le 5$ and partially for $n=6$. There is also a characteristic specific version of the polynomial $h_n$ which is simpler but has the same property. Let $p=\text{char}\,\f_q$. There is a polynomial $h_{p,n}(X_1,\dots,X_n)\in\f_p[X_1,\dots,X_n]$, depending on $p$, such that if $f=X^n+a_1X^{n-1}+\cdots+a_n\in\f_q[X]$ is irreducible and $h_{p,n}(a_1,\dots,a_n)\ne 0$, then $f$ is normal. We remind the reader that $h_{p,n}$ is not necessarily the reduction of $h_n$ modulo $p$.

The basic idea of our approach is the notion of symmetrization of polynomials. 
Let $F$ be a field, $n\ge 2$ be an integer and $f(X_0,\dots,X_{n-1})\in F[X_0,\dots,X_{n-1}]$. We are interested in finding a symmetric polynomial $g(X_0,\dots,X_{n-1})\in F[X_0,\dots,X_{n-1}]$ such that $f\mid g$; we refer to such a $g$ as a {\em symmetrization} of $f$. Naturally, we require the degree of $g$ to be as low as possible. The relevance of this question will become clear in the next section. 

Let the symmetric group $S_n$ act on $F[X_0,\dots,X_{n-1}]$ by permuting $X_0\cdots,X_{n-1}$. Let $\text{Stab}(f)=\{\sigma\in S_n:\sigma(f)=f\}$ be the stabilizer of $f$ in $S_n$. Let $\mathcal C$ be a system of representatives of the left cosets of $\text{Stab}(f)$ in $S_n$. Then $g=\prod_{\sigma\in\mathcal C}\sigma(f)$ is a symmetrization of $f$. (Note that $g$ is independent of the choice of $\mathcal C$.) To reduced the degree of the symmetrization of $f$, it is sometimes helpful to factor $f$ first and then determine the symmetrization fo each factor.

\section{Symmetrization of $\Delta_n$}

Define 
\begin{equation}\label{2.1}
\Delta_n(X_0,\dots,X_{n-1})=\det\left[
\begin{matrix}
X_0&X_1&\cdots&X_{n-1}\cr
X_{n-1}&X_0&\cdots&X_{n-2}\cr
\vdots&\vdots&\ddots&\vdots\cr
X_1&X_2&\cdots&X_0
\end{matrix}\right]\in\Bbb Z[X_0\dots,X_{n-1}].
\end{equation}
Let $f=X^n+a_1X^{n-1}+\cdots+a_n\in\f_q[X]$ be irreducible, and let $r_0=r, r_1=r^q,\dots,r_{n-1}=r^{q^{n-1}}$ be the roots of $f$ in $\f_{q^n}$. It is well known that $r_0,\dots,r_{n-1}$ are linearly independent over $\f_q$ if and only if $\Delta_n(r_0,\dots,r_{n-1})\ne 0$; see \cite[Lemma~3.51]{Lidl-Niederreiter-FF-1997}.
Note that $\Delta_n(r_0,\dots,r_{n-1})$ cannot be expressed in terms of the coefficients $a_1,\dots, a_n$. However, let $\Sigma_n$ denote a symmetrization of $\Delta_n$. Then $\Sigma_n(r_0,\dots,r_{n-1})$ is a polynomial in $a_1,\dots, a_n$, and moreover, $\Sigma_n(r_0,\dots,r_{n-1})\ne 0$ implies $\Delta_n(r_0,\dots,r_{n-1})\ne 0$. The purpose of this section is to determine $\Sigma_n$.

We can write 
\begin{equation}\label{2.2}
\Delta_n=\prod_{i\in\Bbb Z/n\Bbb Z}\Bigl(\sum_{j\in\Bbb Z/n\Bbb Z}\epsilon_n^{ij}X_j\Bigr),
\end{equation}
where $\epsilon_n=e^{2\pi\sqrt{-1}/n}$. Let 
\begin{equation}\label{2.3}
\Psi_n(X_0,\dots,X_{n-1})=\prod_{i\in(\Bbb Z/n\Bbb Z)^\times}\Bigl(\sum_{j\in\Bbb Z/n\Bbb Z}\epsilon_n^{ij}X_j\Bigr).
\end{equation}
Clearly, $\Psi_n$ is invariant under the action of the Galois group $\text{Aut}(\Bbb Q(\epsilon_n)/\Bbb Q)$. Hence $\Psi_n\in\Bbb Q[X_0,\dots,X_{n-1}]$. Since the coefficients of $\Psi_n$ are integral over $\Bbb Z$, we have $\Psi_n\in\Bbb Z[X_0,\dots,X_{n-1}]$. For the formulas of $\Psi_n$ with $n\le 6$, see the appendix A1. From \eqref{2.2} we have
\begin{equation}\label{2.4}
\Delta_n=\prod_{m\mid n}\Psi_m\Bigl(\sum_{j=0}^{n/m-1}X_{mj},\sum_{j=0}^{n/m-1}X_{1+mj},\dots,\sum_{j=0}^{n/m-1}X_{m-1+mj}\Bigr).
\end{equation}

We treat the symmetric group $S_n$ as the permutation group of $\Bbb Z/n\Bbb Z=\{0,1,\dots,n-1\}$ and $S_{n-1}<S_n$ as the permutation group of $\{1,2,\dots,n-1\}$.
For $a\in(\Bbb Z/n\Bbb Z)^\times$ and $b\in\Bbb Z/n\Bbb Z$, let $\alpha_{a,b}$ be the permutation of $\Bbb Z/n\Bbb Z$ defined by the affine map $x\mapsto ax+b$. Let $G_n=\{\alpha_{a,b}:a\in(\Bbb Z/n\Bbb Z)^\times, b\in\Bbb Z/n\Bbb Z\}=\text{AGL}(1,\Bbb Z/n\Bbb Z)<S_n$ and let $G_n^*=\{\alpha_{a,0}: a\in(\Bbb Z/n\Bbb Z)^\times\}<S_{n-1}$. 

\begin{lem}\label{L2.1} 
We have
\[
\text{\rm Stab}(\Psi_n)=
\begin{cases}
\{\text{\rm id}\}&\text{if}\ n=1,2,\cr
G_n&\text{if}\ n>2.
\end{cases}
\]
\end{lem}

\begin{proof}
Since $\Psi_2(X_0,X_1)=X_0-X_1$, its stabilizer in $S_2$ is the trivial subgroup $\{\text{id}\}$. 

Now assume that $n>2$. For $a\in(\Bbb Z/n\Bbb Z)^\times$, we have
\begin{align*}
\alpha_{a,0}(\Psi_n)\,&=\prod_{i\in(\Bbb Z/n\Bbb Z)^\times}\Bigl(\sum_{j\in\Bbb Z/n\Bbb Z}\epsilon_n^{ij}X_{aj}\Bigr)=\prod_{i\in(\Bbb Z/n\Bbb Z)^\times}\Bigl(\sum_{j\in\Bbb Z/n\Bbb Z}\epsilon_n^{ia^{-1}j}X_{j}\Bigr)\cr
&=\prod_{i\in(\Bbb Z/n\Bbb Z)^\times}\Bigl(\sum_{j\in\Bbb Z/n\Bbb Z}\epsilon_n^{ij}X_{j}\Bigr)=\Psi_n.
\end{align*}
For $b\in\Bbb Z/n\Bbb Z$, we have
\begin{align*}
\alpha_{1,1}(\Psi_n)\,&=\prod_{i\in(\Bbb Z/n\Bbb Z)^\times}\Bigl(\sum_{j\in\Bbb Z/n\Bbb Z}\epsilon_n^{ij}X_{j+1}\Bigr)=\prod_{i\in(\Bbb Z/n\Bbb Z)^\times}\Bigl(\sum_{j\in\Bbb Z/n\Bbb Z}\epsilon_n^{i(j-1)}X_{j}\Bigr)\cr
&=\prod_{i\in(\Bbb Z/n\Bbb Z)^\times}\epsilon_n^{-i}\Bigl(\sum_{j\in\Bbb Z/n\Bbb Z}\epsilon_n^{ij}X_{j}\Bigr)=\epsilon_n^{\sum_{i\in(\Bbb Z/n\Bbb Z)^\times}i}\prod_{i\in(\Bbb Z/n\Bbb Z)^\times}\Bigl(\sum_{j\in\Bbb Z/n\Bbb Z}\epsilon_n^{ij}X_{aj}\Bigr)\cr
&=\epsilon_n^{\sum_{i\in(\Bbb Z/n\Bbb Z)^\times}i}\Psi_n.
\end{align*} 
In the above, since $n>2$, we have
\[
\sum_{i\in(\Bbb Z/n\Bbb Z)^\times}i=0.
\]
Hence $\alpha_{1,1}(\Psi_n)=\Psi_n$. Since $G_n$ is generated by $\alpha_{a,0}$, $a\in(\Bbb Z/n\Bbb Z)^\times$, and $\alpha_{1,1}$, we have $G_n\subset\text{Stab}(\Psi_n)$.

On the other hand, assume that $\sigma\in \text{Stab}(\Psi_n)$. Since $\sigma(\sum_{j=0}^{n-1}\epsilon_n^jX_j)=\sum_{j=0}^{n-1}\epsilon_n^{\sigma^{-1}(j)}X_j$ divides $\sigma(\Psi_n)=\Psi_n=\prod_{i\in(\Bbb Z/n\Bbb Z)^\times}(\sum_{j\in\Bbb Z/n\Bbb Z}\epsilon_n^{ij}X_j)$, there exist $i\in(\Bbb Z/n\Bbb Z)^\times$ and $t\in\Bbb C^\times$ such that 
\[
(\epsilon_n^{\sigma^{-1}(0)},\dots,\epsilon_n^{\sigma^{-1}(n-1)})=t(\epsilon_n^{i\cdot 0},\dots,\epsilon_n^{i(n-1)}).
\]
It follows from here that $t=\epsilon_n^{\sigma^{-1}(0)}$ and $\sigma^{-1}(j)=ij+\sigma^{-1}(0)$ for all $j\in\Bbb Z/n\Bbb Z$. Hence $\sigma\in G_n$.
\end{proof}

Let $\mathcal C_n$ be a system of representatives of left cosets of $\text{Stab}(\Psi_n)$ in $S_n$ and let
\begin{equation}\label{2.5}
\Phi_n=\prod_{\sigma\in\mathcal C_n}\sigma(\Psi_n).
\end{equation}
Then $\Phi_n$ is a symmetrization of $\Psi_n$. Obviously, 
\[
\Phi_1=X_0,\quad \Phi_2=(X_0-X_1)(X_1-X_0). 
\]
For $n>2$, $\mathcal C_n$ is a system of representatives of left cosets of $G_n$ in $S_n$. In this case, we can choose $\mathcal C_n$ to be a system of representatives of left cosets of $G_n^*$ in $S_{n-1}$.
To construct $\mathcal C_n$, we can proceed as follows. Partition $\{1,\dots,n-1\}$ as $P_1\sqcup P_2$, where $P_1=(\Bbb Z/n\Bbb Z)^\times=\{1\le i\le n-1:\text{gcd}(i,n)=1\}$ and $P_2=\{1\le i\le n-1:\text{gcd}(i,n)>1\}$. We denote the permutation group of any $P\subset\{1,\dots,n-1\}$ by $S_P$. Consider a tower of subgroups $G_n^*<S_{P_1}\times S_{P_2}<S_{n-1}$.
Then $S_{P_1\setminus\{1\}}\times S_{P_2}$ is a system of representatives of the left cosets of $G_n^*$ in $S_{P_1}\times S_{P_2}$. Let $\mathcal P=\{P\subset\{1,\dots,n-1\}: |P|=\phi(n)\}$. For each $P\in \mathcal P$, choose any $\sigma_P\in S_{n-1}$ such that $\sigma_P(P_1)=P$. Then $\{\sigma_P:P\in\mathcal P\}$ is a system of representatives of the left cosets of $S_{P_1}\times S_{P_2}$ in $S_{n-1}$. Therefore, we can choose
\begin{equation}\label{Cn}
\mathcal C_n=\{\sigma_P\alpha:\alpha\in S_{P_1\setminus\{1\}}\times S_{P_2}, P\in\mathcal P\},\quad n>2.
\end{equation}
For the formulas of $\Phi_n$ with $1\le n\le 6$, see the appendix A2.

We now determine a symmetrization of 
\begin{equation}\label{2.6}
\Psi_m\Bigl(\sum_{j=0}^{n/m-1}X_{mj},\sum_{j=0}^{n/m-1}X_{1+mj},\dots,\sum_{j=0}^{n/m-1}X_{m-1+mj}\Bigr),
\end{equation}
where $m\mid n$. Partition $\Bbb Z/n\Bbb Z$ into blocks $I_i=\{i+mj:0\le j\le n/m-1\}$, $0\le i\le m-1$ and write $Y_i=\sum_{a\in I_i}X_a$. For $G<S_m$, the wreath product $G \kern1pt \wr S_{n/m}$ is the subgroup of $S_n$ consisting of all permutations that are obtained as follows: first permute the blocks $I_1,\dots, I_m$ using a permutation from $G$; then permute the elements in each block $I_i$ independently. More precisely,
\[
G\wr S_{n/m}=\{(\sigma;\sigma_0,\dots,\sigma_{m-1}):\sigma\in G, \sigma_i\in S_{n/m}, 0\le i\le m-1\},
\]
where $(\sigma;\sigma_0,\dots,\sigma_{m-1})$ maps $i+mj$ to $\sigma(i)+m\sigma_i(j)$. Note that $(\sigma;\sigma_0,\dots,\sigma_{m-1})(Y_i)=Y_{\sigma(i)}$ for all $0\le i\le m-1$.

\begin{lem}\label{L2.2}
Let $m\mid n$ and $Y_i=\sum_{j=0}^{n/m-1}X_{i+mj}$. We have
\[
\text{\rm Stab}(\Psi_m(Y_0,\dots,Y_{m-1}))=\text{\rm Stab}(\Psi_m)\wr S_{n/m},
\]
\end{lem}

\begin{proof}
Let $\sigma\in\text{Stab}(\Psi_m)$ and $\sigma_i\in S_{n/m}$, $0\le i\le m-1$.
Then 
\begin{align*}
(\sigma;\sigma_0,\dots,\sigma_{m-1})(\Psi_m(Y_0,\dots,Y_{m-1}))\,&=\Psi_m(Y_{\sigma(0)},\dots,Y_{\sigma(m-1)})\cr
&=\Psi_m(Y_0,\dots,Y_{m-1}).
\end{align*}
Hence $\text{Stab}(\Psi_m)\wr S_{n/m}\subset \text{Stab}(\Psi_m(Y_0,\dots,Y_{m-1}))$.

On the other hand, assume that $\alpha\in \text{Stab}(\Psi_m(Y_0,\dots,Y_{m-1}))$. We first claim that $\alpha$ permutes the blocks $I_0,\dots,I_{m-1}$, i.e., $\alpha\in S_m\wr S_{n/m}$. Assume the contrary. Then there exist some $0\le i\le m-1$ and some $a,b\in I_i$ such that $\alpha(a)\in I_{i_1}$ and $\alpha(b)\in I_{i_2}$, where $i_1\ne i_2$.  Since $\alpha\in S_n$ fixes $\Psi_m(Y_0,\dots,Y_{m-1})$, we know that there exist $u\in(\Bbb Z/m\Bbb Z)^\times$ and $t\in\Bbb C^\times$ such that
\[
\alpha\Bigl(\sum_{k=0}^{m-1}\epsilon_m^kY_k\Bigr)=t\Bigl(\sum_{k=0}^{m-1}\epsilon_m^{uk}Y_k\Bigr).
\]
Comparing the coefficients of $X_{\alpha(a)}$ and $X_{\alpha(b)}$ in the above gives
\[
\epsilon_m^i(X_{\alpha(a)}+X_{\alpha(b)})=t(\epsilon_m^{ui_1}X_{\alpha(a)}+\epsilon_m^{ui_2}X_{\alpha(b)}),
\]
which is impossible. Hence the claim is proved. Now write $\alpha=(\sigma;\sigma_0,\dots,\sigma_{m-1})\in S_m\wr S_{n/m}$, where $\sigma\in S_m$ and $\sigma_i\in S_{n/m}$, $0\le i\le m-1$.  We have $\Psi_m(Y_0,\dots,Y_{m-1})=\alpha(\Psi_m(Y_0,\dots,Y_{m-1}))=\Psi_m(Y_{\sigma(0)},\dots,Y_{\sigma(m-1)})$. Since $Y_0,\dots,Y_{m-1}$ are independent indeterminates, it follows that $\sigma\in\text{Stab}(\Psi_m)$, i.e., $\alpha\in \text{Stab}(\Psi_m)\wr S_{n/m}$.
\end{proof}

For $m\mid n$, consider a tower of subgroups $\text{Stab}(\Psi_m)\wr S_{n/m}<S_m\wr S_{n/m}<S_n$ and recall that $\mathcal C_m$ is a system of representatives of the left cosets of $\text{Stab}(\Psi_m)$ in $S_m$. It is clear that
\[
\mathcal D_{n,m}:=\{(\sigma;\text{id},\dots,\text{id})\in S_m\wr S_{n/m}:\sigma\in \mathcal C_m\}
\]
is a system of representatives of the left cosets of $\text{Stab}(\Psi_m)\wr S_{n/m}$ in $S_m\wr S_{n/m}$. Let $\mathcal P_{n,m}$ be the set of all unordered partitions of $\{0,1,\dots,n-1\}$ into $m$ parts of size $n/m$. For $\{P_0,\dots,P_{m-1}\}\in\mathcal P_{n,m}$, choose a permutation $\phi_{\{P_0,\dots,P_{m-1}\}}\in S_n$ which maps $I_i$ to $P_i$, $0\le i\le m-1$. Then
\[
\mathcal E_{n,m}:=\{\phi_{\{P_0,\dots,P_{m-1}\}}:\{P_0,\dots,P_{m-1}\}\in\mathcal P_{n,m}\}
\]
is a system of representatives of the left cosets of $S_m\wr S_{n/m}$ in $S_n$. Therefore, 
\[
\{\beta\alpha:\alpha\in\mathcal D_{n,m}, \beta\in\mathcal E_{n,m}\}
\]
is a system of representatives of the left cosets of $S_m\wr S_{n/m}$ in $S_n$. Let
\begin{align}\label{theta}
\Theta_{n,m}:\,&=\prod_{\substack{\alpha\in\mathcal D_{n,m}\cr \beta\in\mathcal E_{n,m}}}\beta\alpha(\Psi_m(Y_0,\dots, Y_{m-1}))\\
&=\prod_{\{P_0,\dots,P_{m-1}\}\in\mathcal P_{n,m}}\phi_{\{P_0,\dots,P_{m-1}\}}(\Phi_m(Y_0,\dots,Y_{m-1})).\nonumber
\end{align} 
Then $\Theta_{n,m}$ is a symmetrization of $\Psi_m(Y_0,\dots,Y_{m-1})$. To see the second equality in \eqref{theta}, note that
\[
\prod_{\alpha\in\mathcal D_{n,m}}\alpha(\Psi_m(Y_0,\dots, Y_{m-1}))=\Bigl(\prod_{\sigma\in\mathcal C_m}\sigma(\Psi_m)\Big)(Y_0,\dots,Y_{m-1})=\Phi_m(Y_0,\dots,Y_{m-1}).
\]  
For the formulas of $\Theta_{n,m}$, $1\le n\le 6$, $m\mid n$, see the appendix A3.

Now 
\begin{equation}\label{Sigma}
\Sigma_n:=\prod_{m\mid n}\Theta_{n,m}\in\Bbb Z[X_0,\dots,X_{n-1}]
\end{equation}
is a symmetrization of $\Delta_n$.


\section{A Sufficient Condition for Normal Polynomials}

Let $s_i$ denote the $i$-th elementary symmetric polynomial in $X_0,\dots,X_{n-1}$. In \eqref{Sigma}, since $\Theta_{n,m}\in\Bbb Z[X_0,\dots,X_{n-1}]$ is a symmetric polynomial, there exists $\theta_{n,m}\in\Bbb Z[s_1,\dots,s_n]$ such that
\begin{equation}\label{Theta}
\Theta_{n,m}(X_0,\dots,X_{n-1})=\theta_{n,m} (s_1,\dots,s_n).
\end{equation}
For the formulas of $\theta_{n,m}$ with $1\le n\le 6$, $m\mid n$, $(n,m)\ne(6,6)$,  see the appendix A4. (For the entirety of $\theta_{5,5}$, $\theta_{6,2}$ and $\theta_{6,3}$, see \cite{Hou-arXiv2212.04978}.) $\theta_{6,6}$ is a polynomial of degree 120 in $s_1,\dots,s_6$ with at most 436140 terms, where 436140 is the number of nonnegative integer solutions of $x_1+2x_2+\cdots+6x_6=120$. The computation of $\theta_{6,6}$ is not impossible but will require enormous computing power.

Let 
\begin{equation}\label{h_n}
h_n=\prod_{m\mid n}\theta_{n,m}.
\end{equation}
Then
\begin{equation}\label{h}
\Sigma_n=h_n(s_1,s_2,\dots,s_n).
\end{equation}
Our main result is the following criterion for the normality of an irreducible polynomial in terms of its coefficients.

\begin{thm}\label{main}
Let $f=X^n+a_1X^{n-1}+\cdots+a_n\in\f_q[X]$ be irreducible. If $h_n(a_1,\dots,a_n)\ne 0$, where $h_n$ is defined in \eqref{h_n}, then $f$ is  normal over $\f_q$.
\end{thm}

\begin{proof}
Let $r\in\f_{q^n}$ be a root of $f$, and let $r_0=r, r_1=r^q,\dots,r_{n-1}=r^{q^{n-1}}$, which are all the roots of $f$. By \eqref{h},
\[
\Sigma_n(-r_0,\dots,-r_{n-1})=h_n(a_1,\dots,a_n)\ne 0.
\]
Since $\Delta_n\mid \Sigma_n$, we have $\Delta_n(-r_0,\dots,-r_{n-1})\ne 0$. Thus $f(-X)$ is a normal polynomial over $\f_q$ and so is $f(X)$.
\end{proof}

The above result holds in a more general setting. Let $F$ be any field. Let $f\in F[x]$ be a separable irreducible polynomial of degree $n$ and let $K$ be its splitting field over $F$. Such an $f$ is called a {\em normal polynomial} over $F$ if its roots form a basis (hence a normal basis) of the $K$ over $F$. 

\begin{lem}\label{L3.2} Let $K/F$ be a finite Galois extension of degree $n$ with Galois group $\text{\rm Aut}(K/F)=\{\sigma_1,\dots,\sigma_n\}$. Then $x_1,\dots,x_n\in K$ are linearly independent over $F$ if and only if
\[
\det\left[
\begin{matrix} 
\sigma_1(x_1)&\sigma_1(x_2)&\cdots &\sigma_1(x_n)\cr
\sigma_2(x_1)&\sigma_2(x_2)&\cdots &\sigma_2(x_n)\cr
\vdots&\vdots&&\vdots\cr
\sigma_n(x_1)&\sigma_n(x_2)&\cdots &\sigma_n(x_n)
\end{matrix}\right]\ne 0.
\]
\end{lem}

\begin{proof}
($\Rightarrow$) 
Assume that $(a_1,\dots,a_n)\in K^n$ is such that 
\[
(a_1,\dots,a_n)\left[
\begin{matrix} 
\sigma_1(x_1)&\sigma_1(x_2)&\cdots &\sigma_1(x_n)\cr
\sigma_2(x_1)&\sigma_2(x_2)&\cdots &\sigma_2(x_n)\cr
\vdots&\vdots&&\vdots\cr
\sigma_n(x_1)&\sigma_n(x_2)&\cdots &\sigma_n(x_n)
\end{matrix}\right]=0.
\]
This means that $a_1\sigma_1+\cdots+a_n\sigma_n=0$ as a function from $K$ to $K$ since $x_1,\dots,x_n$ form a basis of $K$ over $F$. By the linear independence of characters \cite[Ch.VI, Theorem~4.1]{Lang-2002}, we have $a_1=\cdots=a_n=0$.

\medskip
($\Leftarrow$) Assume to the contrary that there exists $b_1,\dots,b_n\in F$, not all $0$, such that $a_1x_1+\cdots+a_nx_n=0$. Then 
\[
\left[
\begin{matrix} 
\sigma_1(x_1)&\sigma_1(x_2)&\cdots &\sigma_1(x_n)\cr
\sigma_2(x_1)&\sigma_2(x_2)&\cdots &\sigma_2(x_n)\cr
\vdots&\vdots&&\vdots\cr
\sigma_n(x_1)&\sigma_n(x_2)&\cdots &\sigma_n(x_n)
\end{matrix}\right]
\left[\begin{matrix}b_1\cr\vdots\cr b_n\end{matrix}\right]=0,
\]
which is a contradiction.
\end{proof}

\begin{thm}\label{T3.3}
Let $F$ be any field and let $f=X^n+a_1X^{n-1}+\cdots+a_n\in F[X]$ be separable and irreducible with a cyclic Galois group over $F$. If $h_n(a_1,\dots,a_n)\ne 0$, where $h$ is defined in \eqref{h_n}, then $f$ is normal over $F$.
\end{thm}

\begin{proof}
Using Lemma~\ref{L3.2}, the proof of is identical to that of Theorem~\ref{main}.
\end{proof}

\medskip
\noindent{\bf Remark.} Let $f=X^n+a_1X^{n-1}+\cdots+a_n=\prod_{i=0}^{n-1}(X-\alpha_i)$. It is not difficult to see that $h_n(a_1,\dots,a_n)\ne 0$ if and only if 
\[
\Delta_n(\alpha_{\sigma(0)},\dots,\alpha_{\sigma(n-1)})\ne 0\quad\text{for all}\ \sigma\in S_n.
\]


\section{A $p$-ary Version}

Let $\text{char}\,\f_q=p$. The polynomial $h_n$ in \eqref{h_n} is independent of $p$. When taking $p$ into consideration, we get a characteristic specific version of $h_n$ which is simpler but has the same property.

Write $n=p^em$, where $p\nmid m$ and let $\varepsilon_m$ be a primitive $m$-th root of unity (in some extension of $\f_p$). The polynomial $\Delta_n$ defined in \eqref{2.1} is now treated as a polynomial over $\f_p$. The $p$-ary version of $\Psi_m$ in \eqref{2.3} is  
\begin{equation}\label{Psi-p}
\Psi_{p,m}=\prod_{i\in(\Bbb Z/m\Bbb Z)^\times}\Bigl(\sum_{j\in\Bbb Z/m\Bbb Z}\varepsilon_m^{ij}X_j\Bigr).
\end{equation}
(Note that $\Psi_m$ is defined only when $p\nmid m$.) In fact, $\Psi_{p,m}$ is the reduction of $\Psi_m$ modulo $p$. To see this claim, let $\frak p$ be a prime of $\Bbb Z[\epsilon_m]$ (the ring of integers of $\Bbb Q(\epsilon_m)$) lying above $p$ and treat $\f_p(\varepsilon_m)$ as $\Bbb Z[\epsilon_m]/\frak p$. 

We have
\begin{align}\label{delta-p}
\Delta_n(X_0,\dots,X_n)\,&=\prod_{i\in\Bbb Z/m\Bbb Z}\Bigl(\sum_{j\in\Bbb Z/m\Bbb Z}\varepsilon_m^{ij}(X_j+X_{j+m}+\cdots+X_{j+(p^e-1)m})\Bigr)^{p^e}\\
&=\Delta_m(Y_0,\dots, Y_{m-1})^{p^e},\nonumber
\end{align}
where $Y_j=\sum_{k=0}^{p^e-1}X_{j+km}$. Moreover,
\begin{align}\label{Delta_mY}
\Delta_m(Y_0,\dots, Y_{m-1})\,&=\prod_{l\mid m}\Psi_{p,l}\Bigl(\;\sum_{j\equiv 0\,\text{(mod\, $l$)}}Y_j,\;\sum_{j\equiv 1\,\text{(mod\, $l$)}}Y_j,\;\dots,\; \sum_{j\equiv l-1\,\text{(mod\, $l$)}}Y_j\Bigr),\\
&=\prod_{l\mid m}\Psi_{p,l}(Z_0,\dots,Z_{l-1}),\nonumber
\end{align}
where
\[
Z_k=\sum_{\substack{0\le j\le m-1\cr j\equiv k\,\text{(mod\, $l$)}}}Y_j=\sum_{\substack{0\le i\le n-1\cr i\equiv k\,\text{(mod\, $l$)}}}X_i.
\]
The computations in Section 3 carry through almost entirely in the case of characteristic $p$. For $l\mid m$, let $\Theta_{p,n,l}$ and $\theta_{p,n,l}$ be the reductions of $\Theta_{n,l}$ and $\theta_{n,l}$ modulo $p$, respectively. Then $\Theta_{p,n,l}$ is a symmetrization of $\Psi_{p,l}(Z_0,\dots,Z_{l-1})$ and $\prod_{l\mid m}\Theta_{p,n,l}$
is a symmetrization of $\Delta_m(Y_0,\dots,Y_{m-1})$. Moreover
\[
\Theta_{p,n,l}(X_0,\dots,X_{n-1})=\theta_{p,n,l} (s_1,\dots,s_n),
\]
Let
\begin{equation}\label{hpn} 
h_{p,n}=\prod_{l\mid m}\theta_{p,n,l}.
\end{equation}
Then we have the following analogue of Theorem~\ref{main}.

\begin{thm}\label{T4.1}
Let $f=X^n+a_1X^{n-1}+\cdots+a_n\in\f_q[X]$ be irreducible, where $\text{\rm char}\,\f_q=p$. If $h_{p,n}(a_1,\dots,a_n)\ne n$, then $f$ is  normal over $\f_q$.
\end{thm}

\begin{rmk}\label{R4.2}\rm 
(i) For $l\mid m$ and $0\le t\le e$, the mod $p$ reduction of $\Theta_{n,p^tl}$ is a power of $\Theta_{p,n,l}$, hence the mod $p$ reduction of $\theta_{n,p^tl}$ is a power of $\theta_{p,n,l}$. Therefore, for $a_1,\dots,a_n\in\f_q$, $h_n(a_1,\dots,a_n)\ne0$ if and only if $h_{p,n}(a_1,\dots,a_n)\ne0$.

\medskip

(ii) Let $f=X^n+a_1X^{n-1}+\cdots+a_n\in\f_q[X]$ be irreducible. A necessary condition for $f$ to be normal over $\f_q$ is that the sum of its roots is nonzero, i.e., $a_1\ne 0$. When $n=p^e$, i.e., $m=1$, we have $h_{p,n}(a_1,\dots,a_n)=a_1$. Thus in this case, $f$ is normal if and only if $a_1\ne 0$. This fact was first proved by Perlis \cite{Perlis-DMJ-1942}.  

\medskip

(iii) Assume that $n$ is a prime different from $\text{char}\,\f_q$. Then \eqref{Delta_mY} becomes 
\[
\Delta_n(X_0,\dots,X_{n-1})=(X_0+\cdots+X_{n-1})\Psi_{p,n}(X_0,\cdots,X_{n-1}).
\]
Further assume that $q$ is a generator of $(\Bbb Z/n\Bbb Z)^\times$. Let $\alpha\in\f_{q^n}\setminus\f_q$. We claim that 
\[
\Psi_{p,n}(\alpha,\alpha^q,\dots,\alpha^{q^{n-1}})\ne0.
\]
Assume the contrary. Then $\sum_{j\in\Bbb Z/n\Bbb Z}\varepsilon_n^{ij}\alpha^{q^j}=0$ for some $i\in(\Bbb Z/n\Bbb Z)^\times$. Raising this equation to the power $q^k$, $k\in\Bbb N$, gives 
\[
0=\sum_{j\in\Bbb Z/n\Bbb Z}\varepsilon_n^{q^kij}\alpha^{q^{j+k}}=\sum_{j\in\Bbb Z/n\Bbb Z}\varepsilon_n^{q^ki(j-k)}\alpha^{q^j}=\varepsilon_n^{-ikq^k}\sum_{j\in\Bbb Z/n\Bbb Z}\varepsilon_n^{q^kij}\alpha^{q^j},
\]
i.e.,
$\sum_{j\in\Bbb Z/n\Bbb Z}\varepsilon_n^{q^kij}\alpha^{q^j}=0$.  Since $q$ generates $(\Bbb Z/n\Bbb Z)^\times$, it follows that that 
\[
\sum_{j\in\Bbb Z/n\Bbb Z}\varepsilon_n^{ij}\alpha^{q^j}=0
\]
for {\em all} $i\in(\Bbb Z/n\Bbb Z)^\times$. Consequently, $\alpha=\alpha^q=\cdots=\alpha^{q^{n-1}}$, which is a contradiction. Now, $\Delta_n(\alpha,\alpha^q,\dots,\alpha^{q^{n-1}})\ne 0$ if and only if $\alpha+\alpha^q+\dots+\alpha^{q^{n-1}}\ne 0$.
Therefore, in this case, $a_1\ne 0$ is also a necessary and sufficient condition for $f$ (irreducible) to be normal over $\f_q$. This result was proved first proved by Pei et. al. \cite{Pei-Wang-Omura-IEEE-IT-1986} for $q=2$.

\medskip

(iv) Chang et. al. \cite{Chang-Truong-Reed-JA-2001} proved the converse of (ii) and (iii) in the following sense: If every degree $n$ irreducible polynomial over $\f_q$ with nonzero trace is normal over $\f_q$, then $n$ is either a power of $p\,(\,=\text{char}\,\f_q)$ or a prime different from $p$ such that $q$ is a generator of $(\Bbb Z/n\Bbb Z)^\times$.

\end{rmk}


\section{Number of Normal Polynomials}

Let $N(q,n)$ denote the number of monic normal polynomials of degree $n$ over $\f_q$. This number is known; see Proposition~\ref{P4.1} below. There are at least two proofs, given in \cite{Akbik-JNT-1992} and \cite[Theorem~3.73]{Lidl-Niederreiter-FF-1997}, respectively. We include a short proof which only uses linear algebra. Let $p=\text{char}\,\f_q$ and $n=p^et$, where $e\ge 0$ and $t>0$ are integers such that $p\nmid t$. Let $X^t-1=Q_1\cdots Q_k$ be the factorization of $X^t-1$ over $\f_q[X]$ and let $d_i=\deg Q_i$. The integers $d_1,\dots,d_k$ are the sizes of the orbits in $\Bbb Z/t\Bbb Z$ under the multiplication by powers of $q$, i.e., the sizes of the $q$-cyclotomic cosets modulo $t$. For $d\mid t$, let $o_d(q)$ denote the order of $q$ in $(\Bbb Z/d\Bbb Z)^\times$. Then the multiset $\{d_1,\dots,d_k\}$ consists of $o_d(q)$ with multiplicity $\phi(d)/o_d(q)$ for all $d\mid t$. 

\begin{prop}\label{P4.1}
In the above notation, we have
\[
N(q,n)=\frac 1n q^{(p^e-1)t}\prod_{i=1}^k(q^{d_i}-1)=\frac 1n q^{(p^e-1)t}\prod_{d\mid t}(q^{o_d(q)}-1)^{\phi(d)/o_d(q)}.
\]
\end{prop}

\begin{proof}
Let $\rho$ be the Frobenius of $\f_{q^n}/\f_q$, that is, $\rho(x)=x^q$ for all $x\in\f_{q^n}$. Then $\rho$ is an $\f_q$-linear map from $\f_{q^n}$ to itself. The minimal polynomial of $\rho$ over $\f_q$ is $X^n-1=\prod_{i=1}^kQ_i^{p^e}$, so the elementary divisors of $\rho$ are $Q_i^{p^e}$, $1\le i\le k$. Therefore, the matrix of $\rho$ is similar to
\[
M:=\left[
\begin{matrix} M_1\cr &\ddots\cr &&M_k\end{matrix}\right],
\]
where $M_i$ is the companion matrix of $Q_i^{p^e}$. We know that $x\in\f_{q^n}$ is a root of a normal polynomial of degree $n$ over $\f_q$ if and only if $g(\rho)(x)\ne 0$ for all $g\in\f_q[X]$ with $g\mid X^n-1$ and $g\ne X^n-1$. Identify $\f_{q^n}$ with $\f_q^n$ which is compatible with the above matrix $M$. Then $N(q,n)=|\mathcal X|/n$, where
\[
\mathcal X=\{x\in\f_q^n: g(M)x\ne 0\ \text{for all $g\in\f_q[X]$ with $g\mid X^n-1$ and $g\ne X^n-1$}\}.
\]
For 
\[
x=\left[\begin{matrix}x_1\cr\vdots\cr x_k\end{matrix}\right]\in\f_q^n,
\]
where $x_i$ is of length $p^ed_i$, we have
\begin{align*}
&g(M)x\ne 0\ \text{for all $g\mid X^n-1$ with $g\ne X^n-1$}\cr
\Leftrightarrow\ &\Bigl(\frac{X^n-1}{Q_i}\Bigr)(M)x\ne 0\ \text{for all $1\le i\le k$}\cr
\Leftrightarrow\ &\Bigl(\frac{X^n-1}{Q_i}\Bigr)(M_i)x_i\ne 0\ \text{for all $1\le i\le k$}.
\end{align*}
By \cite[Lemma~6.11]{Hou-ams-gsm-2018},
\[
\text{nullity}\,\Bigl(\Bigl(\frac{X^n-1}{Q_i}\Bigr)(M_i)\Bigr)=\text{nullity}\,(Q_i^{p^e-1}(M_i))=\deg Q_i^{p^e-1}=(p^e-1)d_i.
\]
Hence 
\[
|\mathcal X|=\prod_{i=1}^k(q^{p^ed_i}-q^{(p^e-1)d_i})=\prod_{i=1}^k q^{(p^e-1)d_i}(q^{d_i}-1)=q^{(p^e-1)t}\prod_{i=1}^k(q^{d_i}-1).
\]
\end{proof}

Theorem~\ref{T4.1} is a criterion that allows us to detect the normality of irreducible polynomials from their coefficients. However, since the condition in Theorem~\ref{T4.1} is sufficient but not necessary, irreducible polynomials that do not meet the criterion may still be normal. Naturally, one would like know how often ``false negative'' cases can occur. Let $N'(q,n)$ denote the number of monic normal polynomials of degree $n$ over $\f_q$ that are detected by Theorem~\ref{T4.1}. In Table~\ref{Tb1}, we computed $N(q,n)$ for $q\le 19$ and $n\le 6$ and $N'(q,n)$ for $q\le 19$ and $n\le 6$ except for $(q,n)=(5,6),(7,6),(11,6),(13,6),(17,6),(19,6)$. The result is quite surprising. For the range of $(q,n)$ computed, $N'(q,n)=N(q,n)$ with only two exceptions: $(q,n)=(2,6)$ and $(8,6)$.


\begin{table}[ht]
\caption{$N(q,n)$ vs the numbers of polynomials from Theorem~\ref{T4.1}}\label{Tb1}
   \renewcommand*{\arraystretch}{1.2}
    \centering
     \begin{tabular}{c c|c|c|c}
         \hline
         $q$  &  $n$ & $N(q,n)$ & $N'(q,n)$ &\kern-0.2em $\begin{array}{ll}\textstyle\text{Remark~\ref{R4.2}}\vspace{-0.5em}\cr \textstyle\text{(ii) (iii)}\end{array}$ \kern-0.2em \\ \hline
         2 & 1 & 1 &1 &   1 \\ 
           & 2 & 1 &1 &   1 \\
           & 3 & 1 &1 &   1 \\
           & 4 & 2 &2 &   2 \\
           & 5 & 3 &3 &   3 \\
           & 6 & 4 &0 &  \\ 
           \hline
         3 & 1 & 2 &2 &   2 \\ 
           & 2 & 2 &2 &   2 \\
           & 3 & 6 &6 &   6 \\
           & 4 & 8 &8 \\
           & 5 & 32 &32 &   32 \\
           & 6 & 54 &54 & \\ 
           \hline 
         4 & 1 & 3 & 3 &   3 \\ 
           & 2 & 6 & 6 &   6 \\ 
           & 3 & 9 & 9 &      \\
           & 4 & 48 & 48 &   48 \\ 
           & 5 & 135 & 135 &   \\
           & 6 & 288 & 288 &   \\ 
           \hline
         5 & 1 & 4 &4 &   4 \\ 
           & 2 & 8 &8 &   8 \\
           & 3 & 32 &32 &   32 \\
           & 4 & 64 &64 \\
           & 5 & 500 &500 &   500 \\
           & 6 & 1,536 \\ 
           \hline
         7 & 1 & 6 &6 &   6 \\ 
           & 2 & 18 &18 &   18 \\
           & 3 & 72 &72 \\
           & 4 & 432 &432 \\
           & 5 & 2,880 &2,880 &   2,880 \\
           & 6 & 7,776 \\ 
           \hline
         8 & 1 & 7 & 7 &   7 \\ 
           & 2 & 28 & 28 &   28 \\
           & 3 & 147 & 147 &   147 \\
           & 4 & 896 & 896 &   896 \\
           & 5 & 5,733 & 5,733 &   5,733 \\
           & 6 & 37,632 & 35,280 \\         
         \hline
     \end{tabular}
\end{table}

\addtocounter{table}{-1}

\begin{table}[ht]
\caption{continued}
   \renewcommand*{\arraystretch}{1.2}
    \centering
     \begin{tabular}{c c|c|c|c}
         \hline
         $q$  &  $n$ & $N(q,n)$ &  $N'(q,n)$ &\kern-0.2em $\begin{array}{ll}\textstyle\text{Remark~\ref{R4.2}}\vspace{-0.5em}\cr \textstyle\text{(ii) (iii)}\end{array}$ \kern-0.2em \\ \hline
         9 & 1 & 8 & 8 & 8   \\ 
           & 2 & 32 & 32 & 32   \\
           & 3 & 216 & 216 & 216   \\
           & 4 & 1,024 & 1,024 &    \\
           & 5 & 10,240 & 10,240 &    \\
           & 6 & 69,984 & 69,984 &    \\ 
         \hline
         11 & 1 & 10 & 10 & 10   \\ 
           & 2 & 50 & 50 & 50   \\
           & 3 & 400 & 400 & 400   \\
           & 4 & 3,000 & 3,000 &    \\
           & 5 & 20,000 & 20,000 &    \\
           & 6 & 240,000 &  &    \\ 
         \hline
         13 & 1 & 12 & 12 & 12   \\ 
           & 2 & 72 & 72 & 72   \\
           & 3 & 576 & 576 &    \\
           & 4 & 5,184 & 5,184 &    \\
           & 5 & 68,544 & 68,544 & 68,544   \\
           & 6 & 497,664 &  &    \\ 
         \hline  
         16 & 1 & 15 & 15 & 15   \\ 
           & 2 & 120 & 120 & 120   \\
           & 3 & 1,125 & 1,125 &    \\
           & 4 & 15,360 & 15,360 & 15,360   \\
           & 5 & 151,875 & 151,875 &    \\
           & 6 & 2,304,000 & 2,304,000 &    \\ 
         \hline
         17 & 1 & 16 & 16 & 16   \\ 
           & 2 & 128 & 128 & 128   \\
           & 3 & 1,536 & 1,536 & 1,536   \\
           & 4 & 16,384 & 16,384 &    \\
           & 5 & 267,264 & 267,264 & 267,264   \\
           & 6 & 3,536,944 &  &    \\ 
         \hline
         19 & 1 & 18 & 18 & 18   \\ 
           & 2 & 162 & 162 & 162   \\
           & 3 & 1,944 & 1,944 &    \\
           & 4 & 29,160 & 29,160 &    \\
           & 5 & 466,560 & 466,560 &    \\
           & 6 & 5,668,704 &  &    \\ 
         \hline            
     \end{tabular}
\end{table}


\section{Conclusions and Final Remarks}

Theoretically, for each $n\ge 1$, there exists a polynomial $h_n(X_1,\dots,X_n)\in\Bbb Z[X_1,\dots,X_n]$ such that a separable irreducible polynomial $f(X)=X^n+a_1X^{n-1}+\cdots+a_n$ over any field $F$ with a cyclic Galois group is normal over $F$ if $h_n(a_1,\dots,a_n)\ne 0$. When $\text{char}\,F=p>0$, there is also a $p$-ary version $h_{p,n}(X_1,\dots,X_n)\in\f_p[X_1,\dots,X_n]$ with the same property. The polynomials $h_n$ and $h_{p,n}$ are difficult to compute in general. Explicit computation of $h_n$ and $h_{p,n}$ can be achieved only for small $n$. Once $h_n$ ($h_{p,n}$) is determined, it can be applied to separable irreducible polynomials of degree $n$ over any field (any field of characteristic $p$) with a cyclic Galois group. Experiments indicate that the sufficient condition for the normality of polynomials over finite fields provided by $h_{p,n}$ are close to being necessary. 

The following two questions arise naturally from the approach of the paper:

1. If we are interested in normal polynomials over finite fields with few terms, then the corresponding question is to compute $h_{p,n}(a_1,\dots,a_n)$, where most of the coefficients $a_1,\dots,a_n$ are zero. It turns out that there is an indirect way to compute such $h_{p,n}(a_1,\dots,a_n)$ for slightly larger $n$.

2. The polynomial $h_n$ applies to separable irreducible polynomials of degree $n$ over any field with a cyclic Galois group. More generally, we may consider separable irreducible polynomials of degree $n$ with an arbitrary Galois group $G$ of order $n$. Then the analogy of the polynomial $\Delta_n$ is the {\em group determinant} $\Theta_G$ of $G$ considered by Frobenius; see \cite{Frobenius-1896, Johnson-MPCPS-1991}. The factors of $\Theta_G$ correspond to the irreducible characters of $G$, and the symmetrizations of the factors of $\Theta_G$ can be determined at least in theory.

We plan to discuss the above questions in detail in a separate paper.


\section*{Acknowledgments}

The author thanks Neranga Fernando for the helpful discussions and his assistance in computation.



\section*{Appendix}


\noindent A1. $\Psi_n$, $1\le n\le 6$.

\[
\Psi_1=X_0.
\]

\[
\Psi_2=X_0-X_1.
\]

\[
\Psi_3=X_0^2-X_1 X_0-X_2 X_0+X_1^2+X_2^2-X_1 X_2.
\]

\[
\Psi_4=X_0^2-2 X_2 X_0+X_1^2+X_2^2+X_3^2-2 X_1 X_3.
\]

\[
\longequation{
\Psi_5=X_0^4-X_1 X_0^3-X_2 X_0^3-X_3 X_0^3-X_4 X_0^3+X_1^2 X_0^2+X_2^2 X_0^2+X_3^2
   X_0^2+X_4^2 X_0^2+2 X_1 X_2 X_0^2+2 X_1 X_3 X_0^2-3 X_2 X_3 X_0^2-3 X_1
   X_4 X_0^2+2 X_2 X_4 X_0^2+2 X_3 X_4 X_0^2-X_1^3 X_0-X_2^3 X_0-X_3^3
   X_0-X_4^3 X_0+2 X_1 X_2^2 X_0-3 X_1 X_3^2 X_0+2 X_2 X_3^2 X_0+2 X_1
   X_4^2 X_0+2 X_2 X_4^2 X_0-3 X_3 X_4^2 X_0-3 X_1^2 X_2 X_0+2 X_1^2 X_3
   X_0+2 X_2^2 X_3 X_0-X_1 X_2 X_3 X_0+2 X_1^2 X_4 X_0-3 X_2^2 X_4 X_0+2
   X_3^2 X_4 X_0-X_1 X_2 X_4 X_0-X_1 X_3 X_4 X_0-X_2 X_3 X_4
   X_0+X_1^4+X_2^4+X_3^4+X_4^4-X_1 X_2^3-X_1 X_3^3-X_2 X_3^3-X_1 X_4^3-X_2
   X_4^3-X_3 X_4^3+X_1^2 X_2^2+X_1^2 X_3^2+X_2^2 X_3^2+2 X_1 X_2 X_3^2+X_1^2
   X_4^2+X_2^2 X_4^2+X_3^2 X_4^2-3 X_1 X_2 X_4^2+2 X_1 X_3 X_4^2+2 X_2 X_3
   X_4^2-X_1^3 X_2-X_1^3 X_3-X_2^3 X_3-3 X_1 X_2^2 X_3+2 X_1^2 X_2
   X_3-X_1^3 X_4-X_2^3 X_4-X_3^3 X_4+2 X_1 X_2^2 X_4+2 X_1 X_3^2 X_4-3 X_2
   X_3^2 X_4+2 X_1^2 X_2 X_4-3 X_1^2 X_3 X_4+2 X_2^2 X_3 X_4-X_1 X_2 X_3
   X_4.
}
\]

\[
\longequation{
\Psi_6=X_0^2+X_1 X_0-X_2 X_0-2 X_3 X_0-X_4 X_0+X_5
   X_0+X_1^2+X_2^2+X_3^2+X_4^2+X_5^2+X_1 X_2-X_1 X_3+X_2 X_3-2 X_1 X_4-X_2
   X_4+X_3 X_4-X_1 X_5-2 X_2 X_5-X_3 X_5+X_4 X_5.
}
\]   

\bigskip

\noindent A2. $\Phi_n$, $1\le n\le 6$.

\[
\Phi_1=X_0.
\]

\[
\Phi_2=(X_0-X_1)(X_1-X_0).
\]

\[
\Phi_3=X_0^2-X_1 X_0-X_2 X_0+X_1^2+X_2^2-X_1 X_2.
\]

\[
\Phi_4=\prod_{(i_1,i_2,i_3)}\Psi_4(X_0,X_{i_1},X_{i_2},X_{i_3}), 
\]
where
\[
(i_1,i_2,i_3)=(1, 3, 2), (1, 2, 3), (2, 1, 3).
\]

\[
\Phi_5=\prod_{(i_2,i_3,i_4)}\Psi_5(X_0,X_1,X_{i_2},X_{i_3},X_{i_4}), 
\]
where
\[
(i_2,i_3,i_4)=(2, 3, 4), (2, 4, 3), (3, 2, 4), (4, 2, 3), (3, 4, 2), 
(4, 3, 2).
\]

\[
\Phi_6=\prod_{(i_1,i_2,i_3,i_4,i_5)}\Psi_6(X_0,X_{i_1},X_{i_2},X_{i_3},X_{i_4},X_{i_5}), 
\]
where
\begin{align*}
&(i_1,i_2,i_3,i_4,i_5)=\cr
&(1,3,4,5,2),(1,3,5,4,2),(1,4,3,5,2),(1,4,5,3,2),(1,5,3,4,2),(1,5,4,3,2),\cr
&(1,2,4,5,3),(1,2,5,4,3),(1,4,2,5,3),(1,4,5,2,3),(1,5,2,4,3),(1,5,4,2,3),\cr
&(1,2,3,5,4),(1,2,5,3,4),(1,3,2,5,4),(1,3,5,2,4),(1,5,2,3,4),(1,5,3,2,4),\cr
&(1,2,3,4,5),(1,2,4,3,5),(1,3,2,4,5),(1,3,4,2,5),(1,4,2,3,5),(1,4,3,2,5),\cr
&(2,1,4,5,3),(2,1,5,4,3),(2,4,1,5,3),(2,4,5,1,3),(2,5,1,4,3),(2,5,4,1,3),\cr
&(2,1,3,5,4),(2,1,5,3,4),(2,3,1,5,4),(2,3,5,1,4),(2,5,1,3,4),(2,5,3,1,4),\cr
&(2,1,3,4,5),(2,1,4,3,5),(2,3,1,4,5),(2,3,4,1,5),(2,4,1,3,5),(2,4,3,1,5),\cr
&(3,1,2,5,4),(3,1,5,2,4),(3,2,1,5,4),(3,2,5,1,4),(3,5,1,2,4),(3,5,2,1,4),\cr
&(3,1,2,4,5),(3,1,4,2,5),(3,2,1,4,5),(3,2,4,1,5),(3,4,1,2,5),(3,4,2,1,5),\cr
&(4,1,2,3,5),(4,1,3,2,5),(4,2,1,3,5),(4,2,3,1,5),(4,3,1,2,5),(4,3,2,1,5).
\end{align*}

\bigskip

\noindent A3. $\Theta_{n,m}$, $1\le n\le 6$, $m\mid n$.

\[
\Theta_{1,1}=X_0.
\]

\[
\Theta_{2,1}=X_0+X_1.
\]

\[
\Theta_{2,2}=\Phi_2.
\]

\[
\Theta_{3,1}=X_0 + X_1 + X_2.
\]

\[
\Theta_{3,3}=\Phi_3.
\]

\[
\Theta_{4,1}=X_0 + X_1 + X_2 + X_3.
\]

\[
\Theta_{4,2}=\prod_{(i_1,i_2,i_3)}\Phi_2(X_0+X_{i_2},X_{i_1}+X_{i_3}),
\]
where
\[
(i_1,i_2,i_3)=(2,1,3), (1,2,3), (1,3,2).
\]

\[
\Theta_{4,4}=\Phi_4.
\]

\[
\Theta_{5,1}=X_0 + X_1 + X_2 + X_3 + X_4.
\]

\[
\Theta_{5,5}=\Phi_5.
\]

\[
\Theta_{6,1}=X_0 + X_1 + X_2 + X_3 + X_4 + X_5.
\]

\[
\Theta_{6,2}=\prod_{(i_1,i_2,i_3,i_4,i_5)}\Phi_2(X_0+X_{i_2}+X_{i_4}, X_{i_1}+X_{i_3}+X_{i_5}),
\]
where 
\begin{align*}
(i_1,i_2,i_3,i_4,i_5)=\,&(3,1,4,2,5),(2,1,4,3,5),(2,1,3,4,5),(2,1,3,5,4),(1,2,4,3,5),\cr
&(1,2,3,4,5),(1,2,3,5,4),(1,3,2,4,5),(1,3,2,5,4),(1,4,2,5,3).
\end{align*}

\[
\Theta_{6,3}=\prod_{(i_1,i_2,i_3,i_4,i_5)}\Phi_3(X_0+X_{i_3},X_{i_1}+X_{i_4},X_{i_2}+X_{i_5}),
\]
where 
\begin{align*}
(i_1,i_2,i_3,i_4,i_5)=\,
&(2,4,1,3,5),(2,3,1,4,5),(2,3,1,5,4),(1,4,2,3,5),(1,3,2,4,5),\cr
&(1,3,2,5,4),(1,4,3,2,5),(1,2,3,4,5),(1,2,3,5,4),(1,3,4,2,5),\cr
&(1,2,4,3,5),(1,2,4,5,3),(1,3,5,2,4),(1,2,5,3,4),(1,2,5,4,3).
\end{align*}

\[
\Theta_{6,6}=\Phi_6.
\]

\bigskip

\noindent A4. $\theta_{n,m}$, $1\le n\le 6$, $m\mid n$, $(n,m)\ne (6,6)$.

\[
\theta_{1,1}=s_1.
\]

\[
\theta_{2,1}=s_1.
\]

\[
\theta_{2,2}=-s_1^2+4s_2.
\]

\[
\theta_{3,1}=s_1.
\]

\[
\theta_{3,3}=s_1^2-3s_2.
\]

\[
\theta_{4,1}=s_1.
\]

\[
\theta_{4,2}=-s_1^6+8 s_2 s_1^4-16 s_3 s_1^3-16 s_2^2 s_1^2+64 s_2 s_3 s_1-64
   s_3^2.
\]

\[
\theta_{4,4}=s_1^6-8 s_2 s_1^4+4 s_3 s_1^3+20 s_2^2 s_1^2-24 s_4 s_1^2-8 s_2 s_3
   s_1-16 s_2^3-8 s_3^2+64 s_2 s_4.
\]

\[
\theta_{5,1}=s_1.
\]

\[
\theta_{5,5}=
s_1^{24}-30 s_2 s_1^{22}+\cdots
   +7968750 s_2^3 s_3^3 s_4 s_5+8750000 s_2^6
   s_3 s_4 s_5.\ \text{(325 terms)}
\]

\[
\theta_{6,1}=s_1.
\]

\[
\theta_{6,2}=
s_1^{20}-24 s_2 s_1^{18}+\cdots-2097152 s_2^3 s_3 s_5
   s_6+8388608 s_2 s_3 s_4 s_5 s_6.\ \text{(194 terms)}
\]

\begin{align*}
\theta_{6,3}=\,&
s_1^{30}-36 s_2 s_1^{28}+\cdots-1549681956 s_2^3 s_3^3 s_4 s_5 s_6-172186884 s_2^6 s_3 s_4 s_5s_6.\cr
&\text{(1057 terms)}
\end{align*}




\begin{thebibliography}{99}

\bibitem{Akbik-JNT-1992}
S. Akbik, {\it Normal generators of finite fields}, J. Number Theory {\bf 41} (1992),  146 -- 149. 

\bibitem{Chang-Truong-Reed-JA-2001}
Y. Chang, T. K. Truong, I. S. Reed, {\it Normal bases over $\text{\rm GF}(q)$}, J. Algebra {\bf 241} (2001), 89 -- 101.

\bibitem{Gao-thesis-1993}
S. Gao, {\it Normal Bases over Finite Fields}, Ph.D. Dissertation, University of Waterloo, Canada, 1993.

\bibitem{Deuring-MA-1933}
M. Deuring, {\it Galoissche Theorie und Darstellungstheorie}, Math. Ann. {\bf 107}(1933), 140 -- 144.

\bibitem{Frobenius-1896}
G. Frobenius, {\it \"Uber Gruppencharaktere}, Sitzungsber. Preuss. Akad. Wiss. Berlin (1896), 985 -- 1021. Gesammelte Abhandlungen (Springer-Verlag, 1968), pp. 1 -- 37.

\bibitem{Hou-ams-gsm-2018}
X. Hou, {\it Lectures on Finite Fields}, Graduate Studies in Mathematics 190, American Mathematical Society, Providence, RI, 2018.

\bibitem{Hou-arXiv2212.04978}
X. Hou, {\it Normal polynomials over finite fields},  arXiv:2212.04978v1.

\bibitem{Johnson-MPCPS-1991}
K. W. Johnson, {\it On the group determinant}, Math. Proc. Cambridge Philos. Soc. {\bf 109} (1991), 299 -- 311. 

\bibitem{Lang-2002}
S. Lang, {\it Algebra}, Springer, New York, 2002.

\bibitem{Lidl-Niederreiter-FF-1997}
R. Lidl and H. Niederreiter, {\it Finite Fields},  Cambridge University Press, Cambridge, 1997.

\bibitem{Mullen-Panario-HF-2013}
G. L. Mullen and D. Panario (eds),
{\it Handbook of Finite Fields}, Discrete Mathematics and Its Applications, CRC Press, Boca Raton, FL, 2013. 

\bibitem{Pei-Wang-Omura-IEEE-IT-1986}
D. Pei, C. Wang, J. Omura, {\it Normal bases of finite fields $\text{\rm GF}(2^m)$}, IEEE Trans. Inform. Theory {\bf 32} (1986), 285 -- 287.

\bibitem{Perlis-DMJ-1942}
S. Perlis, {\it Normal bases of cyclic fields of prime power degree}, Duke Math. J. {\bf 9} (1942), 507 -- 517.

\end{thebibliography}
\end{document}